\documentclass[11pt]{compositio}
\usepackage{amsfonts, amssymb, amsmath, amsthm, latexsym, enumerate, epsfig, color, mathrsfs, fancyhdr, pdfsync, eurosym, endnotes, stmaryrd}
\usepackage[all]{xy}

\usepackage[english]{babel}

\usepackage[mac]{inputenc}
\usepackage[T1]{fontenc}

\usepackage{hyperref}

\theoremstyle{plain}
\newtheorem{theorem}{Theorem}[subsection]
\newtheorem{lemma}[theorem]{Lemma}
\newtheorem{proposition}[theorem]{Proposition}
\newtheorem{corollary}[theorem]{Corollary}
\newtheorem{maintheorem}{Theorem}[section]

\theoremstyle{definition}
\newtheorem{definition}[theorem]{Definition}
\newtheorem{notation}[theorem]{Notation}
\newtheorem{setting}[maintheorem]{Setting}

\theoremstyle{remark}
\newtheorem{remark}[theorem]{Remark}

\numberwithin{equation}{subsection}

\newcommand{\Z}{\ensuremath{\mathbf{Z}}}
\newcommand{\R}{\ensuremath{\mathbf{R}}}
\newcommand{\A}{\ensuremath{\mathbf{A}}}
\renewcommand{\P}{\ensuremath{\mathbf{P}}}

\newcommand{\Cs}{\ensuremath{\mathscr{C}}}

\newcommand{\Es}{\ensuremath{\mathscr{E}}}
\newcommand{\Fs}{\ensuremath{\mathscr{F}}}
\newcommand{\Hs}{\ensuremath{\mathscr{H}}}
\newcommand{\Js}{\ensuremath{\mathscr{J}}}
\newcommand{\Os}{\ensuremath{\mathscr{O}}}

\newcommand{\Hc}{\ensuremath{\mathcal{H}}}
\newcommand{\Rc}{\ensuremath{\mathcal{R}}}

\newcommand{\Yk}{\ensuremath{\mathfrak{Y}}}

\newcommand{\E}[2]{\ensuremath{\mathbf{A}^{#1,\mathrm{an}}_{#2}}}
\newcommand\ti[1]{\ensuremath{\tilde{#1}}}
\newcommand{\ho}{\ensuremath{\hat{\otimes}}}
\newcommand{\an}{\textrm{an}}
\newcommand{\Hom}{\textrm{Hom}}

\newcommand{\la}{\langle}
\newcommand{\ra}{\rangle}

\begin{document}
\setlength{\baselineskip}{0.50cm}	

\title[Continuity and finiteness \textit{via} potential theory]{Continuity and finiteness of the radius of convergence of a $p$-adic differential equation \textit{via} potential theory}

\author{J\'er\^ome Poineau}
\email{jerome.poineau@math.unistra.fr}
\address{Institut de recherche math\'ematique avanc\'ee, 7, rue Ren\'e Descartes, 67084 Strasbourg, France}

\author{Andrea Pulita}
\email{pulita@math.univ-montp2.fr}
\address{D\'epartement de Math\'ematiques, Universit\'e de Montpellier II, CC051, Place Eug\`ene Bataillon, 34095, Montpellier Cedex 5, France}

\thanks{The research for this article was partially supported by the ANR projects Berko and CETHop.}
\date{\today}


\begin{abstract}
We study the radius of convergence of a differential equation on a smooth Berkovich curve over a non-archimedean complete valued field of characteristic~0. Several properties of this function are known: F.~Baldassarri proved that it is continuous (see~\cite{ContinuityCurves}) and the authors showed that it factorizes by the retraction through a locally finite graph (see~\cite{finiteness} and~\cite{finiteness2}). Here, assuming that the curve has no boundary or that the differential equation is overconvergent, we provide a shorter proof of both results by using potential theory on Berkovich curves.
\end{abstract}

\maketitle

\section{Introduction}

Let~$K$ be a non-archimedean complete valued field of characteristic 0. Let~$X$ be a quasi-smooth $K$-analytic curve, in the sense of Berkovich theory. Let~$\Fs$ be a locally free $\Os_{X}$-module of finite type endowed with an integrable connection~$\nabla$.

By the implicit function theorem, in the neighbourhood of any $K$-rational point~$x$, the curve is isomorphic to a disc, and it makes sense to consider the radius~$\Rc(x)$ of the biggest disc on which $(\Fs,\nabla)$ is trivial. Extending the scalars, we may find a rational point above any given point and extend the definition to the whole curve~$X$. 

The radius of convergence is a function that has been actively investigated and its general behaviour is now rather well understood. The main features we are interested in are its continuity and a finiteness property (the fact that the map is controlled by its behaviour on a locally finite graph). Both are known. In the nineties already, G.~Christol and B.~Dwork proved the continuity of the radius of convergence over the skeleton of an annulus (see~\cite{ChristolDwork}). This result was later extended to a continuity statement on affinoid domains of the affine line by F.~Baldassarri and L.~Di Vizio (see~\cite{ContinuityBDV}), then on curves by F.~Baldassarri (see~\cite{ContinuityCurves}). As for the finiteness property, it was recently proven by the authors (see~\cite{finiteness} and~\cite{finiteness2}).  The proofs of all these results are quite long and involved, which led us to believe that it was worth finding shorter ones, even to the expense of restricting the setting: here, we assume that the curve is boundary-free (\textit{e.g.} the analytification of an algebraic curve) or that the connection is overconvergent. We have tried to make our paper as self-contained as possible.

Our techniques rely, for the main part, on potential theory on Berkovich curves, as developed by A.~Thuillier in his thesis~\cite{TheseThuillier}. We will provide a reminder in section~\ref{sec:potential} and then use freely the notions of Laplacian operator, harmonic and sub-harmonic functions, etc. As regards \mbox{$p$-adic} differential equations, we will only need to know how the radius of convergence behaves on an interval inside a disc or an annulus: it is continuous, concave, piecewise $\log$-linear with slopes that are rational numbers whose denominators are bounded by the rank of~$\Fs$ (see theorem~\ref{thm:Kedlaya} and the discussion that precedes for attribution of the results).

Let us say a few words about the strategy of the proof. In the case where~$X$ is an analytic domain of the affine line, with coordinate~$t$, the ring~$\Os(X)$ may be endowed with the usual derivation $d=\mathrm{d}/\mathrm{d}t$. If the sheaf~$\Fs$ is free on~$X$, then, together with its connection~$\nabla$, it comes from a differential module~$(M,D)$. The radius of convergence of~$(M,D)$ may then be computed as the radius of convergence of a power series, hence by the usual formula $\liminf_{n} (|f_{n}|^{-1/n})$, for some $f_{n} \in \Os(X)$. It is here that potential theory enters the picture: by general arguments, the logarithm of the resulting function, or more precisely its lower semicontinuous envelope, is super-harmonic. 

Unfortunately, this proof cannot be directly adapted to the case of a general smooth curve, even locally, for lack of a canonical derivation. Still, using geometric arguments, around any point~$x$ of type~2, we manage to find a suitable derivation and use it to show that the logarithm of the radius of convergence of~$(\Fs,\nabla)$ coincides with a super-harmonic function on some affinoid domain~$Y$ of~$X$ containing~$x$ (\textit{i.e.} outside a finite number of branches emanating from~$x$). This will be sufficient for our purposes.

The rest of the proof relies on general properties of super-harmonic functions. We will use in a crucial way the fact that their Laplacians are Radon measures. More precisely, together with the fact that the non-zero slopes of the logarithm of the radius of convergence are bounded below in absolute value, this property implies that, around any point of type~2 of~$X$, there may only be a finite number of directions along which this radius is not constant. Using the same strategy as in~\cite{finiteness}, we will deduce that it is locally constant outside a locally finite subgraph~$\Gamma$ of~$X$.

To prove that the radius of convergence is continuous, it is now enough to show that it is continuous on~$\Gamma$. This will follow from another general property of super-harmonic functions: their restrictions to segments are continuous at points of type~2, 3 and~4.

Let us finally point out that the restriction to boundary-free curves (or overconvergent connections) is due to an inherent limitation of the methods of potential theory. Indeed, points that lie at the boundary of the space behave as if some directions out of them were missing and, at those points, the Laplacian of a function (which is a weighted sum of all outer derivatives) carries too little information for us to use.

\medskip

\noindent\textbf{Acknowledgments} 

We are very grateful to Amaury Thuillier for his comments and helpful discussions on the subject.

\begin{setting}\label{setting}
For the rest of the article, we fix the following: $K$~is a complete non-archimedean valued field of characteristic~0, $X$~is a quasi-smooth $K$-analytic curve\footnote{Quasi-smooth means that $\Omega_{X}$ is locally free, see~\cite[2.1.8]{RSSen}. This corresponds to the notion called ``rig-smooth'' in the rigid analytic setting.} endowed with a weak triangulation~$S$, $\Fs$~is a locally free $\Os_{X}$-module of finite type endowed with an integrable connection~$\nabla$. 
\end{setting}

\section{Definitions}

To define the radius of convergence of~$(\Fs,\nabla)$, one needs to understand precisely the geometry of the curve~$X$. We find it convenient to use A.~Ducros's notion of triangulation (see~\cite{RSSen}), which enables to cut the curve into simple pieces. Our definition of radius of convergence actually depends on the choice of such a (weak) triangulation~$S$ on~$X$. We will not carry out the construction in every detail but content ourselves with the basic definitions and properties. We refer to section~2 of our previous paper~\cite{finiteness2} for a more thorough exposition.

\subsection{Triangulations}

Our reference for this part is A.~Ducros's manuscript~\cite{RSSen} and especially chapter~4.

Let us first recall that a connected analytic space is called a virtual open disc (resp. annulus) if it becomes isomorphic to a union of open discs (resp. annuli) over an algebraically closed valued field.

\begin{notation}
For any subset~$Y$ of~$X$, we denote~$Y_{[2,3]}$ its subset of points of type~2 or~3.
\end{notation} 

\begin{definition}\label{defi:triangulation}
A locally finite subset~$S$ of~$X_{[2,3]}$ is said to be a weak triangulation of~$X$ if any connected component of~$X\setminus S$ is a virtual open disc or annulus.
\end{definition}

Note that a weak triangulation may be empty (\textit{e.g.} in the case of a disc or an annulus).

It is possible to associate a skeleton~$\Gamma_{S}$ to a weak triangulation~$S$ by considering the union of the skeletons of the connected components of $X\setminus S$ that are virtual annuli. It is a locally finite subgraph of~$X_{[2,3]}$.

One of the main results of A.~Ducros's manuscript~\cite{RSSen} is the existence of a triangulation\footnote{A.~Ducros's definition is actually stronger than ours since he requires the connected components of~$X\setminus S$ to be relatively compact. Any triangulation is a weak triangulation.} on any quasi-smooth curve. For the rest of the article, we assume that~$X$ is endowed with a weak triangulation~$S$. 

For any complete valued extension~$L$ of~$K$, the weak triangulation~$S$ may be canonically extended to a weak triangulation~$S_{L}$ of~$X_{L}$.

\subsection{Distances}

In the following, we will need to measure distances on the curve~$X$, or at least on some segments inside it. We will explain quickly how this may be done by recalling the rough lines of A.~Ducros's notion of gauge (``toise'' in French, see~\cite[1.6.1]{RSSen}). This construction is not new and several equivalent ones may be found in the literature (see~\cite[section~2.7]{BR} in the case of the line over an algebraically closed field or~\cite[section~2.2]{TheseThuillier} in the general case, for instance).

We first introduce notations for discs and annuli that will also be useful in the rest of the text.

\begin{notation}
Let~$\E{1}{K}$ be the affine analytic line with coordinate~$t$. Let~$L$ be a complete valued extension of~$K$ and $c\in L$. For $R>0$, we set 
\[D_{L}^+(c,R) = \big\{x\in \E{1}{L}\, \big|\, |(t-c)(x)|\le R\big\}\] 
and  
\[D_{L}^-(c,R) = \big\{x\in \E{1}{L}\, \big|\, |(t-c)(x)|<R\big\}.\]
For $R_{1},R_{2}$ such that $0 < R_{1} \le R_{2}$, we set
\[C_{L}^+(c;R_{1},R_{2}) = \big\{x\in \E{1}{L}\, \big|\, R_{1}\le |(t-c)(x)|\le R_{2}\big\}.\] 
For $R_{1},R_{2}$ such that $0 < R_{1} < R_{2}$, we set
\[C_{L}^-(c;R_{1},R_{2}) = \big\{x\in \E{1}{L}\, \big|\, R_{1} < |(t-c)(x)| < R_{2}\big\}.\] 
\end{notation}

\begin{definition}\label{defi:modulus}
The modulus of the closed annulus $C_{L}^+(c;R_{1},R_{2})$ is defined by
\[\textrm{Mod}(C_{L}^+(c;R_{1},R_{2})) = \frac{R_{2}}{R_{1}}.\]
It is independent of the coordinate~$t$ on the annulus.
\end{definition}

The modulus is the basic tool that helps define distances on curves. To explain the idea in a simple case, let us assume, for a short moment, that~$K$ is algebraically closed and that~$X$ is the affine analytic line~$\E{1}{K}$. In this case, any segment $I \subset X_{[2,3]}$ is the skeleton of a closed annulus~$I^\sharp$ and we may set $\ell(I) = \log(\textrm{Mod}(I^\sharp))$. This defines a gauge on~$X_{[2,3]}$. 

Since moduli of annuli are invariant under Galois action, we may consider virtual annuli instead of annuli and relax the hypothesis that the field~$K$ be algebraically closed. The method may actually be extended to curves, by cutting the segments into a finite number of pieces whose interiors lie inside the affine line. Using this kind of arguments, A.~Ducros shows that there exists a canonical gauge~$\ell$ on~$X_{[2,3]}$ in the general case (see~\cite[proposition~3.4.19]{RSSen}).

In what follows, every time we need to measure a segment (to speak of linear or $\log$-linear maps, to compute derivatives, etc.), we will use this canonical gauge.

\subsection{Radius of convergence}\label{sec:radius}

We are now ready to define the radius of convergence of $(\Fs,\nabla)$.

\begin{definition}\label{def:DxS}
Let~$x \in X$. Let~$L$ be a complete valued extension of~$K$ such that $X_{L}$ contains an $L$-rational point~$\ti{x}$ over~$x$. We denote $D(\ti{x},S_{L})$ the biggest open disc centred at~$\ti{x}$ that is contained in $X_{L}\setminus S_{L}$, \textit{i.e.} the connected component of $X_{L}\setminus \Gamma_{S_{L}}$ that contains~$\ti{x}$. 
\end{definition}

\begin{definition}\label{def:radius}
Let~$x$ be a point in~$X$ and~$L$ be a complete valued extension of~$K$ such that $X_{L}$ contains an $L$-rational point~$\ti{x}$ over~$x$. Let us consider the pull-back~$(\ti{\Fs},\ti{\nabla})$ of~$(\Fs,\nabla)$ on $D(\ti{x},S_{L}) \simeq D_{L}^-(0,R)$. We denote~$\Rc'_{S}(x,(\Fs,\nabla))$ the radius of the biggest open disc centred at~$0$ on which $(\ti{\Fs},\ti{\nabla})$ is trivial and $\Rc_{S}(x,(\Fs,\nabla)) = \Rc'_{S}(x,(\Fs,\nabla))/R$.\footnote{If~$\ti D$ denotes the biggest open disc centred at~$\ti x$ on which $(\ti{\Fs},\ti{\nabla})$ is trivial, then $\Rc_{S}(x,(\Fs,\nabla))$ may also be defined as the modulus of the annulus $D(\ti{x},S_{L}) \setminus \ti{D}$ (with an obvious generalisation of definition~\ref{defi:modulus}).}
\end{definition}

The definition is independent of the choices of~$L$ and~$\ti{x}$ and invariant by extension of the base field~$K$.

\begin{notation}
For any complete valued extension~$L$ of~$K$, we denote by $\pi_{L} : X_{L} \to X$ the natural projection.
\end{notation}

\begin{lemma}\label{lem:basechange}
Let~$L$ be a complete valued extension of~$K$. For any $x\in X_{L}$, we have
\[\Rc_{S_{L}}(x,\pi_{L}^*(\Fs,\nabla)) = \Rc_{S}(\pi_{L}(x),(\Fs,\nabla)).\]
\end{lemma}

Let us now explain how the function behaves with respect to changing triangulations. Let~$S'$ be a weak triangulation of~$X$ that contains~$S$. Let $x \in X$. Let~$L$ be a complete valued extension of~$K$ such that $X_{L}$ contains an $L$-rational point~$\ti{x}$ over~$x$. Inside~$X_{L}$, the disc $D(\ti{x},S'_{L})$ is included in $D(\ti{x},S_{L}) \simeq D_{L}^-(0,R)$. Let~$R'$ be its radius as a sub-disc of $D_{L}^-(0,R)$ and set $\rho_{S',S}(x) = R'/R \in (0,1]$. It is also the modulus of the semi-open annulus $D(\ti{x},S_{L}) \setminus D(\ti{x},S'_{L})$. Remark that the map~$\rho_{S',S}$ is constant and equal to~1 on~$S$, and even~$\Gamma_{S}$. It is now easy to check that
\begin{equation}\label{eq:rhoS'S}
\Rc_{S'}(x,(\Fs,\nabla)) = \min \left(\frac{\Rc_{S}(x,(\Fs,\nabla))}{\rho_{S',S}(x)},1\right).
\end{equation}

It is possible to describe in a very concrete way the behaviour of the function~$\rho_{S',S}$ on discs and annuli, hence on the whole~$X$. We deduce the following result. 

\begin{lemma}\label{lem:rhoS'S}
The map $x\in X \mapsto \rho_{S',S}(x)$ is continuous on~$X$, locally constant outside the skeleton~$\Gamma_{S'}$ and piecewise $\log$-linear on~$\Gamma_{S'}$ with slopes~$0$ or~$\pm 1$.
\end{lemma}

\subsection{Analytic domains of the affine line}\label{sec:affineline}

Assume that~$X$ is an analytic domain of the affine line~$\A^{1,\an}_{K}$. The choice of a coordinate~$t$ on~$\A^{1,\an}_{K}$ provides a global coordinate on~$X$ and it seems natural to use it in order to measure the radii of convergence. We will call ``embedded'' the radii we define in this setting. 

Let us first give a definition that does not refer to any triangulation.

\begin{definition}\label{def:multiradiusemb}
Let~$x$ be a point of~$X$ and~$L$ be a complete valued extension of~$K$ such that $X_{L}$ contains an $L$-rational point~$\ti{x}$ over~$x$. Let $D(\ti{x},X_{L})$ be the biggest open disc centred at~$\ti{x}$ that is contained in~$X_{L}$.

Let us consider the pull-back~$(\ti{\Fs},\ti{\nabla})$ of~$(\Fs,\nabla)$ on $D(\ti{x},X_{L})$. We denote $\Rc^\mathrm{emb}(x,(\Fs,\nabla))$ the radius of the biggest open disc centered at~$\ti{x}$, measured using the coordinate~$t$ on~$\A^{1,\an}_{L}$, on which $(\ti{\Fs},\ti{\nabla})$ is trivial.
\end{definition}

The definition of~$\Rc^\mathrm{emb}(x,(\Fs,\nabla))$ only depends on the point~$x$ and not on~$L$ or~$\ti{x}$. 

\medskip

Let us now state a second definition that takes into account the weak triangulation~$S$ of~$X$.

\begin{definition}\label{def:multiradiusembT}
Let~$x$ be a point of~$X$ and~$L$ be a complete valued extension of~$K$ such that $X_{L}$ contains an $L$-rational point~$\ti{x}$ over~$x$. As in definition~\ref{def:DxS}, consider $D(\ti{x},S_{L})$, the biggest open disc centred at~$\ti{x}$ that is contained in $X_{L}\setminus S_{L}$. We denote~$\rho_{S}(x)$ its radius, measured using the coordinate~$t$ on~$\A^{1,\an}_{L}$. 

Let us consider the pull-back~$(\ti{\Fs},\ti{\nabla})$ of~$(\Fs,\nabla)$ on $D(\ti{x},S_{L})$. We denote $\Rc^\mathrm{emb}_{S}(x,(\Fs,\nabla))$ the radius of the biggest open disc centered at~$\ti{x}$, measured using the coordinate~$t$ on~$\A^{1,\an}_{L}$, on which $(\ti{\Fs},\ti{\nabla})$ is trivial.
\end{definition}

Once again, the definitions of~$\rho_{S}(x)$ and~$\Rc^\mathrm{emb}_{S}(x,(\Fs,\nabla))$ are independent of the choices of~$L$ and~$\ti{x}$. 

\medskip

The radii we have just defined may easily be linked to the one we introduced in definition~\ref{def:radius}. For the second radius, we have
\begin{equation}\label{eq:radiusembT}
\Rc_{S}(x,(\Fs,\nabla)) = \frac{\Rc^\textrm{emb}_{S}(x,(\Fs,\nabla))}{\rho_{S}(x)}.
\end{equation}

Assume that~$X$ is not the affine line and let~$S_{0}$ be its smallest weak triangulation. We have
\begin{equation}\label{eq:radiusembT0}
\Rc_{S_{0}}(x,(\Fs,\nabla)) = \frac{\Rc^\textrm{emb}_{S_{0}}(x,(\Fs,\nabla))}{\rho_{S_{0}}(x)} =  \frac{\Rc^\textrm{emb}(x,(\Fs,\nabla))}{\rho_{S_{0}}(x)}.
\end{equation}

The different radii satisfy similar properties thanks to the following result.

\begin{lemma}\label{lem:rho}
The map $x\in X \mapsto \rho_{S}(x)$ is continuous on~$X$, locally constant outside the skeleton~$\Gamma_{S}$ and piecewise $\log$-linear on~$\Gamma_{S}$ with slopes~$0$ or~$\pm 1$.
\end{lemma}

\subsection{Computation in coordinates}

We now present a concrete way to compute the radius of convergence. Consider an open disc $D = D^-(0,R)$ endowed with the empty weak triangulation and choose a coordinate~$t$ on it. Endow $\Os(D)$ with the usual derivation $d = \mathrm{d}/\mathrm{d}t$. Assume that~$\Fs$ is free of rank~$m$ on~$D$.  In this case, the connection~$\nabla$ on~$D$ may be given by a matrix $G \in M_{m}(\Os(D))$. 

Let $x\in D(K)$. We can compute the Taylor series of a fundamental solution matrix in the neighbourhood of the point~$x$:
\[\sum_{n\ge 0} \frac{G_{n}(x)}{n!}\, (t-t(x))^n,\]
where $G_{0} = \textrm{Id}$, $G_{1}=G$ and, for every $n\ge 1$, $G_{n+1} = d(G_{n}) + G_{n}\, G$. Then the radius of convergence at~$x$ may be computed by the following formula:
\begin{equation}\label{eq:formularadius1}
\Rc_{\emptyset}(x,(\Fs,\nabla)) = \min\left( \frac{1}{R}\, \liminf_{n\ge 1} \left(\left|\frac{G_{n}(x)}{n!}\right|^{-\frac{1}{n}}\right),1\right).
\end{equation}
Since the matrices~$G_{n}$ stay the same if we enlarge the field~$K$, the formula actually holds for any point~$x$ of~$D$.

Even more generally, assume that~$X$ is an analytic domain on the affine line different from the affine line and that~$\Fs$ is free on it. Choose a coordinate~$t$ on~$\E{1}{K}$ and endow $\Os(X)$ with the usual derivation $d = \mathrm{d}/\mathrm{d}t$. We may define a sequence~$(G_{n})_{n\ge 0}$ of matrices as above and check that, for any~$x\in X$, we have
\begin{equation}\label{eq:RdT0}
\Rc_{S_{0}}(x,(\Fs,\nabla)) = \min\left( \frac{1}{\rho_{S_{0}}(x)}\, \liminf_{n\ge 1} \left(\left|\frac{G_{n}(x)}{n!}\right|^{-\frac{1}{n}}\right),1\right),
\end{equation}
where~$S_{0}$ is the smallest weak triangulation of~$X$.

\medskip

Let us now return to the case of the open disc $D = D^-(0,R)$ as above. In general, the sheaf~$\Fs$ need not be free on it (unless~$K$ is maximally complete, see~\cite{Lazard}). On the other hand, it is free on any disc $D^+(0,r)$, with $r<R$, since $\Os(D^+(0,r))$ is principal. For any $r\in [0,R)$ and any $n\ge 0$, we denote~$G_{r,n}$ the matrix associated to the restriction of~$\nabla^n$ to~$D^+(0,r)$. We set
\begin{equation}\label{eq:Rr}
\Rc_{\emptyset,r}(x,(\Fs,\nabla)) = \min\left(\frac{1}{R}\, \liminf_{n\ge 1} \left(\left|\frac{G_{r,n}(x)}{n!}\right|^{-\frac{1}{n}}\right),\frac{r}{R}\right).
\end{equation}
Then, it is easy to check that the map $r\in [0,R) \mapsto \Rc_{\emptyset,r}(x,(\Fs,\nabla))$ is non-decreasing and that, for any $x\in D$, we have
\begin{equation}\label{eq:R1}
\Rc_{\emptyset}(x,(\Fs,\nabla)) = \lim_{r\to R^-} \Rc_{\emptyset,r}(x,(\Fs,\nabla)).
\end{equation}

\medskip

Let us now finally turn back to the general case where~$X$ is a quasi-smooth curve. Let $x\in X$. Let~$L$ be a complete valued extension of~$K$ such that $X_{L}$ contains an $L$-rational point~$\ti{x}$ over~$x$. The radius of convergence~$\Rc_{S}(x,(\Fs,\nabla))$ is computed by pulling back~$(\Fs,\nabla)$ to~$(\ti\Fs,\ti\nabla)$ on~$D(\ti{x},S_{L})$. It follows from the definition that 
\begin{equation}
\Rc_{S}(x,(\Fs,\nabla)) = \Rc_{\emptyset}(x,(\ti\Fs,\ti\nabla)),
\end{equation} 
where the latter radius is computed on~$D(\ti{x},S_{L})$ endowed with the empty weak triangulation. Hence the formulas of the preceding paragraph may be used in the general case.

The following result may now be easily proven.

\begin{theorem}\label{thm:Cauchypadique}
For any $x\in X(K)$, we have $\Rc_{S}(x,(\Fs,\nabla)) >0$. In particular, the function $\Rc_{S}(\cdot,(\Fs,\nabla))$ is constant in the neighbourhood of any rational point of~$X$.
\end{theorem}
\begin{proof}
We may use the setting of formula~(\ref{eq:R1}). Let $r\in (0,R)$. We denote~$\|.\|$ the norm on~$D^+(0,r)$. It is enough to show that the sequence $\big((\|G_{r,n}\| / |n !|)^{1/n}\big)_{n\ge 1}$ is bounded. It is well known that the sequence $(|n!|^{1/n})_{n\ge 1}$ converges, hence it is enough to show that $(\|G_{r,n}\|^{1/n})_{n\ge 1}$ is bounded. For any $n\ge 1$, we have $G_{r,n+1} = d(G_{r,n}) + G_{r,n}\, G_{r,1}$, hence $\|G_{r,n+1}\| \le \max(\|d\|, \|G_{r,1}\|) \, \|G_{r,n}\|$. We deduce that, for any $n\ge 1$, we have $\|G_{r,n}\|^{1/n} \le \max(\|d\|, \|G_{r,1}\|)$. 
\end{proof}

\medskip

Let us now try to carry out computations similar to those of formula~(\ref{eq:formularadius1}) around an arbitrary point of~$X$. Let~$U$ be an analytic domain of~$X$ on which the sheaves~$\Fs$ and~$\Omega_{X}$ are free (such an analytic domain exists in the neighbourhood of any point). Let~$d$ be a derivation on~$\Os(U)$. Let~$G$ be the matrix associated to the connection~$(\Fs,\nabla)$. For every $x\in U$, we can now define 
\begin{equation}\label{eq:radiusd}
\Rc^d(x,(\Fs,\nabla)) = \liminf_{n\ge 1} \left(\left|\frac{G_{n}(x)}{n!}\right|^{-\frac{1}{n}}\right).
\end{equation}
However, it is not clear how this relates to the radius of convergence~$\Rc_{S}(x,(\Fs,\nabla))$. We will study this question later (see corollary~\ref{cor:Rdbranch} and remark~\ref{rem:Rdx}).

\section{The result}

Now that we have made precise the meaning of radius of convergence of~$(\Fs,\nabla)$ at any point of the curve~$X$ (see definition~\ref{def:radius}), we start investigating its properties.

\subsection{Statement}

Let us state precisely the result we are interested in. To this end, we need to extend the notion of $\log$-linearity beyond~$X_{[2,3]}$.

\begin{definition}\label{defi:loglinear}
Let~$J$ be a segment of~$X$. A map $f : J \to \R$ is said to be linear if it is continuous and linear on the interior~$\mathring{J}$ of~$J$.

A map $f : \Gamma\to \R$ on a locally finite subgraph~$\Gamma$ of~$X$ is said to be piecewise linear if~$\Gamma$ may be covered by a locally finite family~$\Js$ of segments such that, for any~$J \in \Js$, the restriction of the map~$f$ to~$\mathring{J}$ is linear.

A map with values in~$\R_{+}^*$ is said to be $\log$-linear if its logarithm is linear.
\end{definition}

\begin{theorem}\label{thm:continuousandfinite}
The map
\[x \in X \mapsto \Rc_{S}(x,(\Fs,\nabla)) \in \R_{+}^*\]
satisfies the following properties:
\begin{enumerate}[\it i)]
\item it is continuous;
\item it is locally constant outside a locally finite subgraph~$\Gamma$ of~$X$;
\item its restriction to~$\Gamma$ is piecewise $\log$-linear and, for any connected subgraph~$\Gamma_{c}$ of~$\Gamma$, its slopes on~$\Gamma_{c}$ are rational numbers of the form~$\pm m/i$,  with $m\in\Z$, \mbox{$1\le i \le \textrm{rk}\, (\Fs_{|\Gamma_{c}})$}. 
\end{enumerate}
\end{theorem}

\begin{remark}\label{rem:factorisation}
Let us enlarge~$\Gamma$ to a locally finite subgraph~$\Gamma'$ of~$X$ such that
\begin{enumerate}[\it i)]
\item $\Gamma'$ contains $\Gamma_{S}$ ;
\item $\Gamma'$ meets every connected component of~$X$ ;
\item for any connected component~$V$ of~$X\setminus S$, the graph~$\Gamma'\cap V$ is convex.  
\end{enumerate}
In this case, there is a natural continuous retraction $X \to \Gamma'$ and the map $\Rc_{S}(\cdot,(\Fs,\nabla))$ factorizes by it. 
\end{remark}

In the following, we will give a new proof of the theorem assuming that~$X$ is boundary-free or that $(\Fs,\nabla)$ is overconvergent.

\medskip

As we mentioned in the introduction, the contents of theorem~\ref{thm:continuousandfinite} already appeared in the literature:
\begin{enumerate}[\it i)]
\item The continuity property on the skeleton of an annulus is due to G.~Christol and B.~Dwork (see~\cite[th\'eor\`eme~2.5]{ChristolDwork}). It has been extended to affinoid domains of the affine line by F.~Baldassarri and L.~Di Vizio (see~\cite{ContinuityBDV}) and to general curves by F.~Baldassarri (see~\cite{ContinuityCurves}). His setting is actually slightly less general than ours, but his result extends easily. The second author also proved the continuity of all the radii of convergence (\textit{i.e.} all the slopes of the Newton polygon) on affinoid domains of the affine line by another method (see~\cite{finiteness}). It has been extended to curves by both authors (see~\cite{finiteness2}).
\item The local constancy outside a locally finite subgraph has been proven for all the radii of convergence on an affinoid domain of the affine line by the second author (see~\cite{finiteness}) and then extended to general curves by both authors (see~\cite{finiteness2}).
\item The last property has been proven by \'E.~Pons (in a weaker form) for the skeleton of an annulus (see~\cite[th\'eor\`eme~2.2]{PonsPadova}) and by F.~Baldassarri for an interval inside the skeleton of a curve (see~\cite[corollary~6.0.6]{ContinuityCurves}). In the case of the skeleton of an annulus again, K.~Kedlaya extended it to the other radii of convergence (see~\cite[theorem~11.3.2]{pde}). 
\end{enumerate}

In the sequel, we will use the fact that the result is already known for the restriction of the radius to intervals inside discs and annuli. For future reference, let us write it down explicitly.

\begin{theorem}\label{thm:Kedlaya}
Assume that~$X$ is an open or closed annulus (possibly a disc). Fix a coordinate~$t$ on~$X$ and let $d = \textrm{d}/\textrm{d}t$ be the usual derivation on~$\Os(X)$. Assume that~$(\Fs,\nabla)$ comes from a global differential module~$(M,D)$ on~$(\Os(X),d)$.

If~$X$ is an annulus, let~$J$ be its skeleton. If~$X$ is a disc, pick a point $c \in X(K)$ and let~$J$ be the interval in~$X$ that joins~$c$ to its boundary. 

Then, the restriction of the map $\Rc^\mathrm{emb}(\cdot,(\Fs,\nabla))$ to~$J$ is concave, continuous and piecewise $\log$-linear with slopes that are rational numbers of the form~$m/i$, with $m\in \Z$ and $1\le i \le \textrm{rk}\, (\Fs)$.
\end{theorem}

In most of the literature (see~\cite[theorem~11.3.2]{pde}, for instance), the result is actually stated with the generic radius of convergence. Since the relation between the two radii is well understood (see~\cite[proposition~9.7.5]{pde} or~\cite[section~3.3]{finiteness}), this actually causes no harm.

\subsection{Potential theory}\label{sec:potential}

In this section, we assume that the absolute value of~$K$ is non-trivial and that~$X$ is strictly $K$-analytic and boundary-free. We briefly introduce non-archi\-me\-dean potential theory, as developed in Amaury Thuillier's manuscript~\cite{TheseThuillier}. This will be our main tool in the proof of theorem~\ref{thm:continuousandfinite}.

\medskip

Let us recall that any $K$-analytic space~$Y$ in the sense of V.~Berkovich has a boundary~$\partial Y$ and an interior $\textrm{Int}(Y) = Y \setminus \partial Y$ (see \cite{rouge}, section~2.5 for the affinoid case and the discussion before proposition~3.1.3 for the general one). For instance, the closed disc~$D^+(c,R)$, with $R>0$ has boundary~$\{\eta_{c,R}\}$ and the closed annulus~$C^+(c;R_{1},R_{2})$, with $0<R_{1}\le R_{2}$ has boundary $\{\eta_{c,R_{1}},\eta_{c,R_{2}}\}$. Any open disc or annulus and, more generally, any open subset of the affine line, any open subset of the analytification of an algebraic variety is boundary-free. In this section, $X$~is assumed to be boundary-free and all of its open subsets will also be. 

\medskip

Let us now turn to potential theory. Topologically speaking, Berkovich analytic curves can be reconstructed from their finite subgraphs and we will first explain what it looks like on the latter. On a finite metrized graph~$\Gamma$, one may define
\begin{itemize}
\item smooth functions: they are the continuous piecewise linear functions;
\item the Laplacian of a smooth function~$f$: it is the finite measure
\begin{equation}\label{eq:ddc}
\textrm{dd}^c(f) = \sum_{p\in \Gamma} \big(\sum_{\vec v \in T_{p}\,\Gamma} m_{\vec v}\, \mathrm{d}_{\vec v}f(p)\big) \delta_{p},
\end{equation}
where~$T_{p}\,\Gamma$ denotes the set of directions out of~$p$, $m_{\vec v}$ is a weight, $\mathrm{d}_{\vec v}f(p)$~denotes the outer derivative of~$f$ at~$p$ and $\delta_{p}$~denotes the Dirac measure at~$p$.
\end{itemize}
One may push further this line of thought and define harmonic functions (those for which \mbox{$\textrm{dd}^c(f)=0$}), super-harmonic functions (those for which $\textrm{dd}^c(f)\le 0$) and sub-harmonic functions.

To give a rough idea of what is going on, let us give a few examples. Saying that a smooth function~$f$ on a segment~$[a,b]$ is harmonic on~$(a,b)$ is equivalent to saying that its slope never changes. We deduce that such a function~$f$ is linear on~$[a,b]$, hence determined by its values~$f(a)$ and~$f(b)$ at the boundary. Conversely, any prescribed values at the boundary may be realized by a smooth function on~$[a,b]$ which is harmonic on~$(a,b)$. The analogues of both statements hold for arbitrary finite graphs (see~\cite[proposition~1.2.15]{TheseThuillier}). To put it in other words, in this context, the Dirichlet problem admits a unique solution.

The same kind of arguments show that smooth functions on segments that are super-harmonic in the interior correspond to concave functions.

\medskip

Building on those ideas in the case of graphs, A.~Thuillier managed to develop a full-fledged potential theory on Berkovich analytic curves which is quite similar to the complex one. We briefly review here the definitions he introduces (see~\cite[sections~2 and~3]{TheseThuillier}).

The first step is to extend the notion of smooth and harmonic functions. Let~$\Yk$ be a semi-stable formal scheme whose generic fiber~$Y$ identifies to an affinoid domain of~$X$. In this situation, one defines a skeleton~$\Gamma(\Yk)$\footnote{This skeleton is actually denote~$S(\Yk)$ in~\cite{TheseThuillier}. We changed the notation to avoid the confusion with a triangulation.}, which is a finite subgraph of~$Y$ that contains~$\partial Y$, and a retraction $\tau_{\Yk} : Y \to \Gamma(\Yk)$. Formula~(\ref{eq:ddc}) may now be made more precise by choosing for the weight~$m_{\vec v}$ the residual degree of the field on which the direction corresponding to~$\vec v$ is defined. In particular, if~$K$ is algebraically closed, it is always~1. Let~$H(Y)$ be the pull-back by~$\tau_{\Yk}$ of the set of smooth functions on~$\Gamma(\Yk)$ that are harmonic outside~$\partial Y$.

If~$Y$ is an arbitrary strictly $K$-affinoid domain of~$X$, by the semi-stable reduction theorem, one may carry out the previous construction after passing to a finite Galois extension~$K'/K$ and then consider the invariants under $\mathrm{Gal}(K'/K)$. The resulting set $H(Y) \subset \Cs^0(Y,\R)$, the set of harmonic functions on~$Y$, depends only on~$Y$. 

For any open subset~$U$ of~$X$, one may now set $\Hc_{X}(U) = \varprojlim H(Y)$, where the limit is taken over the strictly $K$-affinoid domains~$Y$ of~$U$. This defines the sheaf~$\Hc_{X}$ of harmonic functions on~$X$. 

If~$Y$ is a $K$-affinoid domain of~$X$, we extend the previous definition of~$H$ by setting $H(Y) = \Gamma(Y\setminus \partial Y,\Hc_{X})\, \cap\, \Cs^0(Y,\R)$. In this general case, the Dirichlet problem also admits a unique solution: the restriction map $H(Y) \to \Hom(\partial Y,\R)$ is bijective. 

Let~$U$ be an open subset of~$X$. The $\R$-vector space $A^0(U)$ of smooth functions on~$U$ consists of the continuous functions $f \in \Cs^0(U,\R)$ for which there exists a locally finite covering of~$U$ by $K$-affinoid domains~$Y$ such that $f_{|Y} \in H(Y)$. If we denote by~$A^1(U)$ the $\R$-vector space of real measures on~$U$ whose support is a locally finite subset of $U_{[2,3]}$ (which is denoted by~$I(U)$ in~\cite{TheseThuillier}), we may naturally extend the Laplacian operator defined by formula~(\ref{eq:ddc}) to a map
\[\textrm{dd}^c : A^0(U) \to A^1(U).\]
As we expect, its kernel is nothing but $\Hc_{X}(U)$.

This operator sends $A_{c}^0(U)$ to $A_{c}^1(U)$, where the subscript~$c$ indicates a compactness condition on the support, and induces a map between their duals
\[\textrm{dd}^c : D^0(U) \to D^1(U).\]
It may be useful to remark that the set~$D^0(U)$ of currents of degree~0 is naturally isomorphic to $\Hom(U_{[2,3]},\R)$.

\medskip

In a more restricting setting, we would also like to mention the book~\cite{BR} by M.~Baker and R.~Rumely where a potential theory on the line over an algebraically closed field is developed. Of course, it is equivalent to A.~Thuillier's.

\medskip

In this text, we will be especially interested in super-harmonic functions (see~\cite[sections~3.1.2 and~3.4]{TheseThuillier} and~\cite[chapter~8]{BR}). We will recall the results we need. Let us begin with the definition (see~\cite[d\'efinition~3.1.5]{TheseThuillier}).

\begin{definition}
Let~$U$ be an open subset of~$X$. We say that a map $u : U \to \R\cup\{+\infty\}$ is pre-super-harmonic if, for any strictly $k$-affinoid~$Y$ of~$U$ and any harmonic function~$h$ on~$Y$, the following condition holds:
\[(u_{|\partial Y} \ge h_{|\partial Y}) \implies (u_{|Y} \ge h).\]
The map~$u$ is said to be super-harmonic if, moreover, it is lower semicontinuous and identically equal to~$+\infty$ on no connected component of~$U$.
\end{definition}

Super-harmonic functions may be characterised by a non-positivity property of their Laplacians (see~\cite[proposition~3.4.4 and th\'eor\`eme~3.4.12]{TheseThuillier} or~\cite[theorem~8.19]{BR}). 

\begin{theorem}\label{thm:Radon}
Let~$U$ be an open subset of~$X$.
\begin{enumerate}[\it i)]
\item Let~$f\in A^0(U)$. The smooth function~$f$ is super-harmonic if, and only if, $\mathrm{d}\mathrm{d}^c(f) \le 0$.
\item Let $T\in D^0(U)$. The current~$T$ is a super-harmonic function if, and only if, \mbox{$\mathrm{d}\mathrm{d}^c(T) \le 0$} (as a current of degree~1). In this case, $\mathrm{d}\mathrm{d}^c(T)$ is a non-positive Radon measure.
\end{enumerate}
\end{theorem}

Let us introduce the basic examples of harmonic and super-harmonic functions (see \cite[propositions~2.3.20 and~3.1.6]{TheseThuillier}).

\begin{proposition}
Let~$U$ be an open subset of~$X$. Let $f\in \Os(U)$.
\begin{enumerate}[\it i)]
\item If~$f$ is invertible on~$U$, then $-\log(|f|)$ is harmonic on~$U$.
\item If~$f$ vanishes identically on no connected component of~$U$, then $-\log(|f|)$ is super-harmonic on~$U$.
\end{enumerate}
\end{proposition}

Since we will use it later, let us mention that we already encountered an example of harmonic function in section~\ref{sec:affineline}.

\begin{lemma}\label{lem:rhoh}
Let~$C$ be an open disc or annulus endowed with the empty weak triangulation. The map $\log(\rho_{\emptyset})$ is harmonic on~$C$. 
\end{lemma}
\begin{proof}
Assume that~$C$ is the open disc $D^-(c,R)$. Then $\rho_{\emptyset} \equiv R$ and the result is obvious.

Assume that~$C$ is the open annulus $C^-(c;R_{1},R_{2})$ with coordinate~$t$. Then, for every~\mbox{$x\in C$}, we have $\rho_{\emptyset}(x) = |(t-c)(x)|$ and the result follows from the proposition.
\end{proof}

In general, it is easy to check that a map of the form~$\rho_{S}$ on an arbitrary analytic domain of the affine line is super-harmonic, but not necessarily harmonic.

\medskip

We now state some properties of super-harmonic functions. It is well-known that a concave map on a segment is left and right-differentiable in the interior of this segment. The next proposition is the analogue of this fact for more general finite 1-dimensional graphs. It appears in~\cite[proposition~8.24]{BR} in the case of the line and may be generalised to the case of curves. 

\begin{proposition}\label{prop:shcontinuous}
Let~$U$ be a connected open subset of~$X$ and $u : U \to \R\cup\{+\infty\}$ be a pre-super-harmonic function on~$U$ that is not identically equal to~$+\infty$. Let~$\Gamma$ be a subgraph of~$U$. For any point~$p\in \Gamma$ of type~2, 3 or~4 and any direction~$\vec{v} \in T_{p}\, \Gamma$, the directional derivative~$\mathrm{d}_{\vec v}f(p)$ exists and is finite. In particular, the restriction of~$f$ to~$\Gamma$ is continuous at any point of type~2, 3 or~4.
\end{proposition}

Let us quote the other properties of super-harmonic functions that we will need. They come from~\cite[proposition~3.1.8]{TheseThuillier}, except for the last property, which may be deduced from the preceding proposition (see also~\cite[proposition~8.26]{BR}).

Recall that, for any topological space~$U$, the lower semicontinuous regularization~$f^*$ of a map $f : U \to (-\infty,+\infty]$ which is locally bounded below is defined by
\[\forall x\in U, f^*(x) = \liminf_{y\to x} f(y).\]

\begin{proposition}\label{prop:propsh}
\begin{enumerate}[\it i)]
\item Let~$U$ be an open subset of~$X$. If $f$ and $g$ are super-harmonic functions on~$U$, then $\min(f,g)$ is super-harmonic on~$U$ and, for any~$\lambda,\mu \ge 0$, $\lambda f +\mu g$ is super-harmonic of~$U$.
\item Let~$U$ be a connected open subset of~$X$. Let $(f_{n})_{n\ge 0}$ be a sequence of super-harmonic functions on~$U$ which is locally bounded below. Put $f = \liminf_{n} (f_{n})$. Then either $f$ is identically equal to~$+\infty$ on~$U$ or the lower semicontinuous regularisation~$f^*$ of~$f$ is super-harmonic. Furthermore, for every $x\in U\setminus U(K)$, we have $f^*(x)=f(x)$.  
\end{enumerate}
\end{proposition}

\begin{corollary}\label{cor:Rdsh}
Let~$U$ be a connected open subset of~$X$ on which the sheaf~$\Fs$ is free and~$d$ be a derivation on~$\Os(U)$. The map~$\log(\Rc^{d})^*$ is either identically equal to~$+\infty$ or super-harmonic on~$U$.
\end{corollary}
\begin{proof}
With the notations of formula~(\ref{eq:radiusd}), for every~$x$ in~$U$, we have 
\[\log(\Rc^{d})(x,(\Fs,\nabla)) =  \liminf_{n\ge 1} \left(-\frac{1}{n}\, \log\left(\left|\frac{G_{n}(x)}{n!}\right|\right)\right).\]
We only need to check that the above sequence is locally bounded from below. Let~$V$ be a compact subset of~$U$. Arguing as in the proof of theorem~\ref{thm:Cauchypadique}, we show that the sequence $\big((\|G_{n}\|_{V} / |n !|)^{1/n}\big)_{n\ge 1}$ is bounded.
\end{proof}

\begin{corollary}\label{cor:logRT0sh}
Assume that~$X$ is an open disc or annulus endowed with the empty weak triangulation on which~$\Fs$ is free. Then the map~$\log(\Rc_{\emptyset})$ is super-harmonic on~$X$. 
\end{corollary}
\begin{proof}
Set the notations as in formula~(\ref{eq:RdT0}). We have
\[\log(\Rc_{\emptyset}) = \min (\log(\Rc^d) - \log(\rho_{\emptyset}),0).\]

By proposition~\ref{prop:propsh}, \textit{ii}), and lemma~\ref{lem:rhoh}, the map $\min(\log(\Rc^d) - \log(\rho_{\emptyset}),0)$ may only fail to be lower semicontinuous at rational points. By theorem~\ref{thm:Cauchypadique}, it is actually constant in the neighbourhood of rational points, hence lower semicontinuous. Thus, we have
\[\log(\Rc_{\emptyset}) = \min (\log(\Rc^d)^* - \log(\rho_{\emptyset}),0).\]
By the preceding corollary, lemma~\ref{lem:rhoh} and proposition~\ref{prop:propsh}, \textit{i)}, it is super-harmonic.
\end{proof}

\subsection{Proof of theorem~\ref{thm:continuousandfinite} in the boundary-free case}

Let us begin with some reductions. Thanks to lemma~\ref{lem:basechange}, we may extend the base field and assume that~$K$ is algebraically closed, non-trivially valued and maximally complete and that~$X$ is strictly $K$-affinoid. We will do so in the rest of the section.

We begin with the simple case of a differential module on an open disc or annulus. We will follow the strategy that was implemented by the second author in~\cite[proof of theorem~2.14]{finiteness}.

\begin{notation}
For $c\in K$ and $R>0$, we denote~$\eta_{c,R}$ the unique point of the Shilov boundary of the disc~$D^+(c,R)$. We set $\eta_{R} = \eta_{0,R}$.
\end{notation}

\begin{lemma}\label{lem:borddisque}
Assume that~$X$ is a closed disc~$D^+(c,R)$ endowed with the smallest weak triangulation $S_{0} = \{\eta_{c,R}\}$. Let~$J$ be the segment $[c,\eta_{c,R}]$. The map $\Rc_{S_{0}}(\cdot,(\Fs,\nabla))$ is piecewise $\log$-linear on~$J$. All its slopes are non-positive. If the last slope is zero, then it is constant on the open disc~$D^-(c,R)$.
\end{lemma}
\begin{proof}
On the closed disc~$D^+(c,R)$, we have $\Rc^\mathrm{emb}(\cdot,(\Fs,\nabla)) = R\, \Rc_{S_{0}}(\cdot,(\Fs,\nabla))$. Hence theorem~\ref{thm:Kedlaya} also holds for the map $\Rc_{S_{0}}(\cdot,(\Fs,\nabla))$, which proves the first statement.

By theorem~\ref{thm:Cauchypadique}, the map $\Rc_{S_{0}}(\cdot,(\Fs,\nabla))$ is constant in the neighbourhood of~$c$. Hence, its first slope on~$J$ is~0. Since it is concave, all its slopes are non-positive.

Now, assume that the last slope of the map $\Rc_{S_{0}}(\cdot,(\Fs,\nabla))$ on the segment~$J$ is~0. The previous argument shows that it is indeed constant on~$J$. Let $y \in D^-(c,R)$. Since~$K$ is maximally complete, there exists $d \in D^-(c,R)(K)$ and $r \in [0,R)$ such that $y = \eta_{d,r}$ (see~\cite[section~1.4.4]{rouge}). The segment $[d,\eta_{d,R}] = [d,\eta_{c,R}]$ meets~$J$ in a neighbourhood of its end. Hence the last slope of the map on this segment is~0 too. Using the same arguments as before, we show that the map is actually constant on the whole segment. Hence $\Rc_{S_{0}}(y,(\Fs,\nabla)) = \Rc_{S_{0}}(\eta_{c,R},(\Fs,\nabla))$.
\end{proof}

We now deal with the case of an open disc or annulus. Recall that we assumed that~$K$ is maximally complete. Hence, by~\cite{Lazard} (which actually deals with the case of a disc, but the result for an annulus follows), any locally free sheaf on such a space is actually free.

We say that a property holds for almost every element of a set~$E$ if it holds for every element of~$E$ with a finite number of exceptions.

\begin{lemma}\label{lem:dirfinite}
Let $C = C^-(0;R_{1},R_{2})$ be an open annulus. Let~$x\in (\eta_{R_{1}},\eta_{R_{2}})$. Let~$\vec{v}_{1}$ and~$\vec{v}_{2}$ be the directions out of~$x$ towards~$\eta_{R_{1}}$ and~$\eta_{R_{2}}$ respectively. Let~$f$ be a super-harmonic function on~$C$. Assume that, for every direction~$\vec v$ out of~$x$, there exists a point~$x_{\vec{v}}$ in the direction of~$\vec v$ such that~$f$ is linear on the segment~$[x,x_{\vec v}]$, with slope~$p_{\vec{v}}$. Assume moreover that there exists~$m>0$ such that, for any direction~$\vec{v}$ different from~$\vec{v}_{1}$ and~$\vec{v}_{2}$, we have $p_{\vec{v}} \in \{0\} \cup [m,+\infty)$. 

Then, for almost every direction out of~$x$, the slope of the map~$f$ is zero.
\end{lemma}
\begin{proof}
Let~$E$ be the set of directions out of~$x$, different from~$\vec{v}_{1}$ and~$\vec{v}_{2}$, on which the slope of the map~$f$ is not zero (hence at least~$m$). Let~$F$ be a finite subset of~$E$. Let~$r$ denote its cardinality. We may assume that there exists~$d$ such that, for any $\vec{v} \in F\cup\{\vec{v}_{1},\vec{v}_{2}\}$, the point~$x_{\vec{v}}$ lies at distance~$d$ from the point~$x$.

Let us consider the continuous function $g : C \to \R$ such that, for any $\vec{v} \in F\cup\{\vec{v}_{1},\vec{v}_{2}\}$, the restriction of~$g$ to~$[x,x_{\vec{v}}]$ is the linear map equal to~1 at~$x$ and~0 at~$x_{\vec{v}}$ and~$g$ is contant outside the complement of those segments. The function~$g$ is smooth with compact support. By definition of the Laplacian operator, we have
\begin{align*}
\la \mathrm{dd}^c(f), g\ra &= \la f, \mathrm{dd}^c(g) \ra\\
&= \bigg\langle f,  - \frac{r+2}{d}\, \delta_{x} + \sum_{\vec{v} \in F\cup\{\vec{v}_{1},\vec{v}_{2}\}} \frac{1}{d}\, \delta_{x_{\vec{v}}}  \bigg\rangle\\
&= \frac{1}{d} \bigg(-(r+2) f(x)  + \sum_{\vec{v} \in F\cup\{\vec{v}_{1},\vec{v}_{2}\}} (f(x) + d\, p_{\vec{v}}) \bigg)\\
&= \sum_{\vec{v} \in F\cup\{\vec{v}_{1},\vec{v}_{2}\}} p_{\vec{v}}\\
&\ge p_{\vec{v}_{1}} + p_{\vec{v}_{2}} + rm.
\end{align*}
Since~$f$ is super-harmonic, by theorem~\ref{thm:Radon}, the current $\mathrm{dd}^c(f)$ is non-positive, which implies that $r \le -(p_{\vec{v}_{1}} + p_{\vec{v}_{2}})/m$. We deduce that the set~$E$ is finite.
\end{proof}

\begin{remark}
It is possible to give another proof of the result using M.~Baker and R.~Rumely's theory. They actually show that a super-harmonic function $f : U \to \R$ is locally of bounded differential variation (see~\cite[theorem~8.19]{BR}): for any~$x\in U$, there exists an open neighbourhood~$V$ and a constant~$B$ such that, for any finite subgraph~$\Gamma$ of~$V$ that contains no points of type~1, we have $|\mathrm{dd}^c(f)(\Gamma)| \le B$. With the notations of the proof of the lemma, this also implies that~$r$ is bounded.
\end{remark}

\begin{proposition}\label{prop:finiteopendisc}
Assume that~$X$ is an open disc or annulus endowed with the empty weak triangulation. Then theorem~\ref{thm:continuousandfinite} holds.
\end{proposition}
\begin{proof}
Assume that~$X$ is the open annulus $C^-(0;R_{1},R_{2})$. This restriction is only a matter of notation and the case of a disc may be handled in the same way.

By theorem~\ref{thm:Kedlaya}, the map $\log(\Rc^\mathrm{emb}(\cdot,(\Fs,\nabla)))$ is continuous and piecewise linear on $J = (\eta_{R_{1}},\eta_{R_{2}})$. We also know that the (archimedean) absolute values of its non-zero slopes in the directions out of a point of~$J$ are uniformly bounded below by a positive constant (the inverse of the rank of~$\Fs$). Moreover, by lemma~\ref{lem:borddisque}, the slopes that do not correspond to the directions~$\vec{v}_{1}$ and~$\vec{v}_{2}$ towards~$\eta_{R_{1}}$ and~$\eta_{R_{2}}$ respectively are non-negative (a minus sign appeared since we now compute the slopes in the other direction). By  formula~(\ref{eq:radiusembT0}) and the explicit description of~$\rho_{\emptyset}$ (constant out of~$J$, log-linear with slope~1 on~$J$; see the proof of lemma~\ref{lem:rhoh}), the map $LR = \log(\Rc_{\emptyset}(\cdot,(\Fs,\nabla)))$ satisfies the same properties.

By corollary~\ref{cor:logRT0sh}, the map~$LR$ is super-harmonic. Let~$x$ be a point of~$J$. By lemma~\ref{lem:dirfinite}, there may only be a finite number of directions out of~$x$ in which the slopes of~$LR$ are non-zero. Recall that, by lemma~\ref{lem:borddisque}, a zero slope correspond to an open disc on which the map~$LR$ is constant. In particular, the map~$LR$ is locally a smooth function, which enables to compute its Laplacian by formula~(\ref{eq:ddc}).

Let~$B_{J}$ be the subset of break-points of~$J$, \textit{i.e.} the set of points~$x$ on~$J$ at which the slope of the map~$LR$ on~$J$ changes. We want to prove that every connected component of~$C\setminus J$ on which the map~$LR$ is not constant branches at a point of~$B_{J}$. Let~$x\in J\setminus B_{J}$. The map~$LR$ has non-negative slopes at~$x$ in the directions outside~$J$ and the slopes in the directions~$\vec{v}_{1}$ and~$\vec{v}_{2}$ balance out. On the other hand, by super-harmonicity, the Laplacian of~$LR$ at the point~$x$, \textit{i.e.} the sum of all the slopes out of~$x$, is a non-positive real number, which forces all the slopes in direction different from~$\vec{v}_{1}$ and~$\vec{v}_{2}$ to be zero.

Let~$D$ be a connected component of $X\setminus J$, necessarily an open disc, and let~$\eta_{D}$ be its boundary point. Let~$x \in D(K)$ and set~$J_{D} = [x,\eta_{D})$. Remark that the map~$LR$ may only have finitely many different slopes on~$J_{D}$ in the neighbourhood of~$\eta_{D}$. Indeed, on~$J_{D}$, it is concave, piecewise linear with slopes that are rational numbers with bounded denominators, and by proposition~\ref{prop:shcontinuous}, it admits a finite derivative at the point~$\eta_{D}$ in the direction of~$D$. In particular, the number of break-points of~$LR$ on~$J_{D}$ is finite.

We now use the argument of the third and fourth paragraphs repeatedly for any connected component of $X\setminus J$ on which~$LR$ is not constant. We need only repeat the process a finite number of times, otherwise we would find an infinite number of breaks on some segment $[\eta_{c,R'_{1}},\eta_{c,R'_{2}}]$ inside a closed sub-disc, which would contradict theorem~\ref{thm:Kedlaya}. This proves that the map~$LR$ is locally constant outside a finite graph~$\Gamma$, which is a finite union of segments of the form $[\eta_{d,R''_{1}},\eta_{d,R''_{2}}]$.

To prove property~\textit{i)} of theorem~\ref{thm:continuousandfinite}, \textit{i.e.} that $\Rc_{\emptyset}(\cdot,(\Fs,\nabla))$ is continuous, it is now enough to prove that it is continuous on~$\Gamma$, hence on a segment of the form $[\eta_{d,R''_{1}},\eta_{d,R''_{2}}]$. This follows from theorem~\ref{thm:Kedlaya}. Property~\textit{iii)} also follows from theorem~\ref{thm:Kedlaya}.
\end{proof}

The following result may be proved by the same arguments.

\begin{corollary}\label{cor:finitegraph}
Fix the setting as in proposition~\ref{prop:finiteopendisc}.

Assume that $X = D^-(0,R)$ and let $J=[0,\eta_{R}]$. Assume that the restriction of the map $\Rc^\mathrm{emb}(\cdot,(\Fs,\nabla))$ to~$J$ is piecewise $\log$-linear with a finite number of slopes. Then the graph~$\Gamma$ of theorem~\ref{thm:continuousandfinite} may be chosen finite in the direction of~$\eta_{R}$: there exists a finite subgraph~$\bar\Gamma$ of the closure $\overline{D^-(0,R)} = D^-(0,R) \cup \{\eta_{R}\}$ of $D^-(0,R)$ such that $\bar\Gamma \cap D^-(0,R) = \Gamma \cap D^-(0,R)$.

Assume that $C = C^-(0;R_{1},R_{2})$ and let $J=[\eta_{R_{1}},\eta_{R_{2}}]$. Let~$R'\in (R_{1},R_{2})$. Assume that the restriction of the map $\Rc^\mathrm{emb}(\cdot,(\Fs,\nabla))$ to~$[\eta_{R'},\eta_{R_{2}})$ is piecewise $\log$-linear with a finite number of slopes. Then the graph~$\Gamma$ of theorem~\ref{thm:continuousandfinite} may be chosen finite in the direction of~$\eta_{R_{2}}$: there exists a finite subgraph~$\bar\Gamma$ of $\overline{C^-(0;R',R_{2})} = C^-(0;R',R_{2}) \cup \{\eta_{R'},\eta_{R_{2}}\}$ such that $\bar\Gamma \cap C^-(0;R',R_{2}) = \Gamma \cap C^-(0;R',R_{2})$. 

The same statement holds if we consider the other end of the annulus.
\end{corollary}

The main issue is now to understand the behaviour of the radius of convergence at a point of the triangulation. We will first consider points of type~2 and study the local structure of the curve~$X$ in the neighbourhood of those points. The main result we use is adapted from A.~Ducros's manuscript~\cite{RSSen} (see the proof of th\'eor\`eme~3.4.1 and also~\cite[theorem~3.2.1]{finiteness2}). 

Let us recall a few definitions from~\cite{RSSen}. A branch roughly corresponds to a direction out of a point (see~\cite[section~1.7]{RSSen} for a precise definition). A section of a branch out of a point~$x$ is a connected open 
subset~$U$ that lies in the prescribed direction and such that~$x$ belongs to the closure~$\bar{U}$ of~$U$ but not to~$U$ itself.

\begin{theorem}[(A.~Ducros)]\label{thm:bonvois}
Let~$x$ be a point of~$X$ of type~2 and~$b$ a branch out of~$x$. There exists an affinoid neighbourhood~$Y$ of~$x$ in~$X$, an affinoid domain~$W$ of~$\P^{1,\textrm{an}}_{K}$ and a finite \'etale map \mbox{$\psi : Y \to W$} such that
\begin{enumerate}[\it i)]
\item $\psi^{-1}(\psi(x))=\{x\}$;
\item almost every connected component of $Y\setminus\{x\}$ is an open unit disc with boundary~$\{x\}$;
\item almost every connected component of $W\setminus\{\psi(x)\}$ is an open unit disc with boundary~$\{\psi(x)\}$;
\item for almost every connected component~$V$ of $Y\setminus\{x\}$, the induced morphism $V \to \psi(V)$ is an isomorphism;
\item the map~$\psi$ induces an isomorphism between a section of~$b$ and a section of~$\psi(b)$. 
\end{enumerate}
\end{theorem}

\begin{corollary}\label{cor:Rdbranch}
Let~$x$ be a point of~$S$ of type~2 and~$b$ a branch out of~$x$. There exists an affinoid neighbourhood~$Y$ of~$x$ in~$X$, a derivation~$d$ on~$\Os(Y)$ and a constant $R>0$ such that
\begin{enumerate}[\it i)]
\item the sheaf~$\Fs$ is free on~$Y$;
\item almost every connected component of $Y\setminus\{x\}$ is an open unit disc on which 
$\Rc_{S}(\cdot,(\Fs,\nabla)) = \min(\Rc^d(\cdot,(\Fs,\nabla))/R, 1) = \min(\Rc^d(\cdot,(\Fs,\nabla))^*/R, 1)$;
\item there exists a section~$U$ of~$b$ which is isomorphic to a semi-open annulus with boundary~$x$ and on which $\Rc_{\emptyset}(\cdot,(\Fs,\nabla)_{|U}) = \min(\Rc^d(\cdot,(\Fs,\nabla))/R, 1) = \min(\Rc^d(\cdot,(\Fs,\nabla))^*/R, 1)$.
\end{enumerate}
\end{corollary}
\begin{proof}
Consider a morphism~$\psi$ as in the previous theorem. Denote~$U$ the section of~$b$ that is mentioned there. Up to shrinking~$Y$ and~$W$, we may assume that~$U$ is a semi-open annulus with boundary~$x$ and that it is the connected component of~$Y\setminus\{x\}$ that lies in the direction associated to~$b$. We may also assume that~$\Fs$ is free on~$Y$.

Let~$t$ denote a coordinate on $W \subset \P^{1,\textrm{an}}_{K}$ and consider the derivation $\mathrm{d}/\mathrm{d}t : \Omega^1_{W} \to \Os_{W}$. Since~$\psi$ is \'etale, it induces a derivation
\[d : \Omega^1_{Y} \simeq \psi^*\Omega^1_{W} \to \psi^*\Os_{W} \to \Os_{Y}\]
on~$Y$. By formula~(\ref{eq:formularadius1}), it satisfies the properties of the statement with $R=|\psi(x)|$, except for the last equalities in the last two items, which is a lower semicontinuity issue.

By proposition~\ref{prop:propsh}, \textit{ii}), the map $\min(\Rc^d(\cdot,(\Fs,\nabla))/R, 1)$ may only fail to be lower semicontinuous at rational points. If it is equal to $\Rc_{S}(\cdot,(\Fs,\nabla))$ or $\Rc_{\emptyset}(\cdot,(\Fs,\nabla)_{|U})$ in the neighbourhood of such a point, by theorem~\ref{thm:Cauchypadique}, there exists a possibly smaller neighbourhood on which it is actually constant, hence lower semicontinuous.
\end{proof}

To be able to use the last point of the previous result, we will need to be able to compare the radius $\Rc_{S}(\cdot,(\Fs,\nabla))$ restricted to some annulus~$C$ in~$X$ to the radius of the restriction $\Rc_{\emptyset}(\cdot,(\Fs,\nabla)_{|C})$. This relies on formula~(\ref{eq:rhoS'S}). 

\begin{lemma}\label{lem:restriction}
Let~$x$ be a point of~$S$. Let~$C$ be an open disc or annulus inside~$X$ such that $\bar{C}\cap S = \{x\}$, where~$\bar{C}$ denotes the closure of~$C$ in~$X$. 
\begin{enumerate}[a)]
\item Assume that~$C$ is an open disc. Then, for any $y\in C$, we have \[\Rc_{\emptyset}(y,(\Fs,\nabla)_{|C}) = \Rc_{S}(y,(\Fs,\nabla)).\]
\item Assume that~$C$ is an open annulus such that $\Gamma_{S}$ meets the skeleton of~$C$. Then, for any $y\in C$, we have \[\Rc_{\emptyset}(y,(\Fs,\nabla)_{|C}) = \Rc_{S}(y,(\Fs,\nabla)).\]
\item Assume that~$C$ is an open annulus such that $\Gamma_{S}$ does not meet the skeleton of~$C$. Identify~$C$ with an annulus $C^-(0;R_{1},R_{2})$, with coordinate~$t$, in such a way that $\lim_{R\to R_{2}^-} \eta_{R} = x$. Then, for any $y\in C$, we have \[\Rc_{\emptyset}(y,(\Fs,\nabla)_{|C}) = \min\left(\frac{R_{2}}{|t(y)|}\, \Rc_{S}(y,(\Fs,\nabla)), 1\right).\]
\end{enumerate}
\end{lemma}

The following result will now allow us to go from one radius to the other to prove the properties we want. It is based on the fact that radii cannot exceed~1. 

\begin{lemma}\label{lem:change}
Let~$x$ be a point of~$S$. Let~$C$ be an open disc or annulus inside~$X$ such that $\bar{C}\cap S = \{x\}$.
\begin{enumerate}[a)]
\item Let~$y \in C$. The restriction of $\Rc_{S}(\cdot,(\Fs,\nabla))$ to~$[y,x)$ admits a limit at~$x$ if, and only if, the restriction of $\Rc_{\emptyset}(\cdot,(\Fs,\nabla)_{|C})$ to~$[y,x)$ admits a limit at~$x$. Moreover, in this case, the limits coincide.
\item Let~$\Gamma$ be a subgraph of~$C$. If~$C$ is an annulus and not a disc, assume that~$\Gamma$ contains its skeleton. The restriction of $\Rc_{S}(\cdot,(\Fs,\nabla))$ to~$C$ is locally constant outside~$\Gamma$ if, and only if, the restriction of $\Rc_{\emptyset}(\cdot,(\Fs,\nabla)_{|C})$ to~$C$ is locally constant outside~$\Gamma$. 
\end{enumerate}
\end{lemma}

Let us now go back to the results we want to prove. We are ready to adapt the proofs of lemma~\ref{lem:dirfinite} and proposition~\ref{prop:finiteopendisc} in the case of a general curve.

\begin{corollary}\label{cor:locfinite}
Let~$x$ be a point of~$S \cap \mathrm{Int}(X)$ of type~2 and let~$C$ be the connected component of~$X$ containing~$x$. The map $\Rc_{S}(\cdot,(\Fs,\nabla))$ is constant on almost every connected component of~$C\setminus\{x\}$. 
\end{corollary}
\begin{proof}
Let~$U$ be an open neighbourhood of~$x$ in~$C$ on which~$\Fs$ and~$\Omega_{X}$ are free. Almost every connected component of~$C\setminus\{x\}$ is a disc that is entirely contained in~$U$. By lemma~\ref{lem:borddisque}, it is enough to prove that, for almost every direction out of~$x$, the slope of~$\Rc_{S}$ at~$x$ in the corresponding direction is~$0$. Thanks to corollary~\ref{cor:Rdbranch}, we may prove the result for the map $\min(\log(\Rc^d)^*,\log(R))$. 

By corollary~\ref{cor:Rdsh}, the map $\min(\log(\Rc^d)^*,\log(R))$ is super-harmonic on~$U$ and, by theorem~\ref{thm:Radon}, the current $\textrm{dd}^c \min(\log(\Rc^d)^*,\log(R))$ is a non-positive Radon measure. In particular, it defines a continuous linear form on the Fr\'echet space~$\Cs^0(U,\R)$: for any compact subset~$V$ of~$U$, there exists~$C_{V}\in \R$ such that, for $f\in\Cs^0(U,\R)$ supported on~$V$, we have 
\[\left|\int_{U} f\, \textrm{dd}^c \min(\log(\Rc^d)^*,\log(R))\right| \le C_{V} \sup_{z\in V}(|f(z)|).\] 
By theorem~\ref{thm:Kedlaya} and corollary~\ref{cor:Rdbranch}, the absolute values of the non-zero slopes of $\min(\log(\Rc^d)^*,\log(R))$ in almost every direction out of~$x$ are uniformly bounded below by a positive constant. 
Hence only a finite number of theses slopes may be different from~$0$. 
\end{proof}

\begin{corollary}\label{cor:finiteopendiscx}
Let~$x$ be a point of~$S \cap \mathrm{Int}(X)$, $b$~a branch out of~$x$ and $C$~an open annulus which is a section of~$b$. By proposition~\ref{prop:finiteopendisc}, there exists a locally finite graph~$\Gamma_{C}$ outside which the map $\Rc_{S}(\cdot,(\Fs,\nabla))_{|C}$ is constant. The graph~$\Gamma_{C}$ may actually be chosen finite in the neighbourhood of~$x$.
\end{corollary}
\begin{proof}
 If~$x$ is a point of type~3 of the triangulation, by~\cite[th\'eor\`eme~3.3.5]{RSSen}, it has a neighbourhood~$U$ which is isomorphic to an open annulus, and the result follows.

We may assume that~$x$ is of type~2, that $C = C^-(0;R_{1},R_{2})$ and that \mbox{$x =\lim_{R' \to R_{2}^-} \eta_{R'}$}. By corollary~\ref{cor:Rdbranch} applied with the branch~$b$ and lemma~\ref{lem:change}, we may prove the result for the map $\min(\log(\Rc^d)^*,\log(R))$. Let $R' \in (R_{1},R_{2})$. By corollary~\ref{cor:finitegraph}, it is enough to prove that the restriction to $I = [\eta_{R'},\eta_{R_{2}})$ of this map has a finite number of slopes. By theorem~\ref{thm:Kedlaya}, it is concave, piecewise $\log$-linear and all its slopes are of the form~$m/i$ with $m\in\Z$ and $1\le i\le \mathrm{rk}(\Fs_{|C})$. By corollary~\ref{cor:Rdsh}, it is also super-harmonic and, by proposition~\ref{prop:shcontinuous}, it has a finite derivative at the point~$x$ in the direction of~$C$. By concavity, the slopes of the map on~$I$ are bounded below, hence there may only be a finite number of them.
\end{proof}

\begin{corollary}\label{cor:boundaryfreefinite}
Assume that~$X$ is boundary-free. Then, there exists a locally finite subgraph~$\Gamma$ of~$X$ outside which the map $\Rc_{S}(\cdot,(\Fs,\nabla))$ is locally constant.
\end{corollary}
\begin{proof}
We may use corollary~\ref{cor:locfinite} for every point of type~2 of the triangulation~$S$. If~$x$ is a point of type~3 of the triangulation, by~\cite[th\'eor\`eme~3.3.5]{RSSen}, it has a neighbourhood~$U$ which is isomorphic to an open annulus. In this case, $U\setminus\{x\}$ is a disjoint union of two annuli.

We are left with a locally finite family~$\Es$ of open discs and annuli on which the function $\Rc_{S}(\cdot,(\Fs,\nabla))$ is not constant. On each member~$E$ of~$\Es$, we may apply proposition~\ref{prop:finiteopendisc} to find a locally finite graph~$\Gamma_{E}$ outside which the map $\Rc_{S}(\cdot,(\Fs,\nabla))_{|E}$ is constant. We check that $\Gamma = \Gamma_{S} \cup \bigcup_{E\in \Es} \Gamma_{E}$ is a graph and that the map $\Rc_{S}(\cdot,(\Fs,\nabla))$ is locally constant on its complement. 

To conclude, we need to prove that the graph~$\Gamma$ is indeed locally finite. It is enough to prove that, for each member~$E$ of~$\Es$ and each point $x\in \partial E$, the graph~$\Gamma_{E}$ is finite in the neighbourhood of~$E$. This is the content of corollary~\ref{cor:finiteopendiscx}.
\end{proof}

Let us finally deal with the continuity of the radius of convergence. Once again, it will follow from corollary~\ref{cor:Rdbranch} and the properties of super-harmonic functions.

\begin{corollary}\label{cor:continuousinterval}
Let~$x$ be a point of~$S \cap \mathrm{Int}(X)$ of type~2 and~$I$ be an interval with end-point~$x$. The restriction of $\Rc_{S}(\cdot,(\Fs,\nabla))$ to~$I$ is continuous at~$x$.
\end{corollary}
\begin{proof}
Let us first explain the strategy of the proof before going into the technical details, which we fear may appear messy. The idea is to use corollary~\ref{cor:Rdbranch} to express the radius~$\Rc_{S}$ as a radius of the form~$\Rc^{d}$, for some derivation~$d$, and then use the continuity property of super-harmonic functions (see proposition~\ref{prop:shcontinuous}). We will use corollary~\ref{cor:Rdbranch} twice: first, after a base change to a field~$L$ containing~$\Hs(x)$, with a branch inside~$\pi_{L}^{-1}(x)$ (we find a first limit $\ell_{1}$ at~$x$, which is $\Rc_{S}(x)$) and second with the branch associated to~$I$ (we find a second limit~$\ell_{2}$ at~$x$, which is the limit of~$\Rc_{S}$ at~$x$ along~$I$). We will conclude that the two limits are the same by a kind of genericity argument: on almost every branch out of~$x$, the two~$\Rc^{d}$'s are equal to~$\Rc_{S}$ (up to a multiplicative constant), hence they coincide and so do their limits at~$x$.  

\medskip

Let us now provide the details of the proof. We may assume that~$I$ is non-trivial. Let~$b$ be the branch out of~$x$ defined by~$I$. Let~$L$ be an algebraically closed complete valued extension of~$\Hs(x)$. Let~$y$ be a point of~$S_{L}$ over~$x$, $I_{0}$~be a non-trivial interval with end-point~$y$ which is contained in~$\pi^{-1}_{L}(x)$ and~$b_{0}$ be the branch out of~$y$ defined by~$I_{0}$. Let us use corollary~\ref{cor:Rdbranch} with~$y$ and~$b_{0}$ to find some~$Y_{0}$, $d_{0}$, $R_{0}$ and~$U_{0}$. 

By proposition~\ref{prop:shcontinuous}, $\Rc^{d_{0}}(z,\pi_{L}^*(\Fs,\nabla))$ tends to~$\Rc^{d_{0}}(y,\pi_{L}^*(\Fs,\nabla))$ when~$z$ tends to~$y$ along~$I_{0}$. By lemma~\ref{lem:basechange}, for any~$z$ in a section of~$c$, we have $\Rc_{S_{L}}(z,\pi_{L}^*(\Fs,\nabla)) = \Rc_{S}(x,(\Fs,\nabla))$. Hence, using lemma~\ref{lem:change}, case~a, we have
\begin{align*}
\min ( \Rc^{d_{0}}(y,\pi_{L}^*(\Fs,\nabla)), R_{0} ) &=  \lim_{z \xrightarrow[I_{0}]{} y} \min( \Rc^{d_{0}}(z,\pi_{L}^*(\Fs,\nabla)) , R_{0})\\
&= \lim_{z \xrightarrow[I_{0}]{} y} R_{0}\, \Rc_{\emptyset}(z,\pi_{L}^*(\Fs,\nabla)_{|U_{0}})\\
&= \lim_{z \xrightarrow[I_{0}]{} y} R_{0}\, \Rc_{S_{L}}(z,\pi_{L}^*(\Fs,\nabla))\\
& = R_{0}\, \Rc_{S}(x,(\Fs,\nabla)).
\end{align*}

Now, let us use corollary~\ref{cor:Rdbranch} with~$x$ and~$b$ to find some~$Y$, $d$, $R$ and~$U$. We may assume that~$I \subset U$. By proposition~~\ref{prop:shcontinuous} and lemma~\ref{lem:change}, we have 
\begin{align*}
\lim_{z \xrightarrow[I]{} x} \Rc_{S}(z,(\Fs,\nabla)) &= \lim_{z \xrightarrow[I]{} x} \Rc_{\emptyset}(z,(\Fs,\nabla)_{|U})\\
&=\lim_{z \xrightarrow[I]{} x} \min ( \Rc^{d}(z,(\Fs,\nabla))/R, 1)\\
&  = \min ( \Rc^{d}(x,(\Fs,\nabla))/R, 1).
\end{align*}

Let us now extend~$d$ to the derivation $d_{L} = d \otimes \mathrm{Id}$ on $\Os(Y_{L}) = \Os(Y)\ho_{K} L$. The conclusion of corollary~\ref{cor:Rdbranch}, forgetting point~\textit{(iii)} that we will not need, still holds with~$y$, $Y_{L}$, $d_{L}$ and~$R$. We have 
\[\Rc^{d_{L}}(y,(\Fs,\nabla)) = \Rc^{d}(x,(\Fs,\nabla)).\]

Let us now remark that there exists an open subset~$V$ of~$X_{L}$ that is both a connected component of $Y_{0}\setminus\{y\}$ and $Y_{L}\setminus\{y\}$ and that satisfies condition~\textit{ii}) of corollary~\ref{cor:Rdbranch} for~$d_{0}$ (with~$R_{0}$) and~$d_{L}$ (with~$R$). In fact, this is the case for almost every connected component of~$C\setminus\{y\}$, where~$C$ denotes the connected component of~$X_{L}$ that contains~$y$. Let~$J$ be a non-trivial interval with end-point~$y$ which is contained in~$V$.

By proposition~~\ref{prop:shcontinuous}, again, we have
\begin{align*}
\min ( \Rc^{d_{0}}(y,\pi_{L}^*(\Fs,\nabla))/R_{0},1) &= \lim_{z \xrightarrow[J]{} y} \min (\Rc^{d_{0}}(z,\pi_{L}^*(\Fs,\nabla))/R_{0},1)\\
&=  \lim_{z \xrightarrow[J]{} y} \Rc_{S}(z,\pi_{L}^*(\Fs,\nabla))\\
&= \lim_{z \xrightarrow[J]{} y} \min (\Rc^{d_{L}}(z,\pi_{L}^*(\Fs,\nabla))/R,1)\\
&= \min( \Rc^{d_{L}}(y,\pi_{L}^*(\Fs,\nabla))/R,1).
\end{align*}

The result follows.
\end{proof}

\begin{remark}\label{rem:Rdx}
Under the conditions of corollary~\ref{cor:Rdbranch}, the preceding proof shows that $\Rc_{S}(x,(\Fs,\nabla)) = \min(\Rc^d(x,(\Fs,\nabla))/R, 1)$. 

Let us point out that we only used the fact that the point~$x$ lies in the interior of~$X$ to ensure that the map~$\Rc^d$ is continuous at~$x$ along any interval. As a consequence, equality holds for a point~$x$ inside the boundary of~$X$, as soon as the continuity property is satisfied. It could be deduced from~\cite[theorem~4.11]{ContinuityBDV} or~\cite[section~5.2]{ContinuityCurves}, for instance.
\end{remark}

\begin{theorem}
Theorem~\ref{thm:continuousandfinite} holds if~$X$ is boundary-free.
\end{theorem}
\begin{proof}
By corollary~\ref{cor:boundaryfreefinite}, there exists a locally finite subgraph~$\Gamma$ of~$X$ outside which the map $\Rc_{S}(\cdot,(\Fs,\nabla))$ is locally constant. This is the content of property~\textit{(ii)}.

It is now enough to prove the continuity of the restriction of $\Rc_{S}(\cdot,(\Fs,\nabla))$ to~$\Gamma$. Let~$x \in \Gamma$. If~$x$ does not belong to the triangulation~$S$, it belongs to an open disc or annulus and the continuity at~$x$ follows from theorem~\ref{thm:Kedlaya}. If~$x$ is a point of~$S$ of type~3, by~\cite[th\'eor\`eme~3.3.5]{RSSen}, it has a neighbourhood~$U$ which is isomorphic to an open annulus, and continuity at~$x$ holds, by the same argument. If~$x$ is a point of~$S$ of type~2, it follows from corollary~\ref{cor:continuousinterval}.

Finally, to prove property~\textit{(iii)}, it is enough to consider the restriction of $\Rc_{S}(\cdot,(\Fs,\nabla))$ to~$\Gamma \cap X\setminus S$, and we may now conclude by theorem~\ref{thm:Kedlaya} again.
\end{proof}

Let us recall that $(\Fs,\nabla)$ is said to be overconvergent if there exist a strictly $K$-analytic curve~$X_{0}$ and a locally free $\Os_{X_{0}}$-module of finite type~$\Fs_{0}$ endowed with an integrable connection~$\nabla_{0}$ such that
\begin{enumerate}[\it i)]
\item $X$ embeds in $\textrm{Int}(X_{0})$ ;
\item $(\Fs_{0},\nabla_{0})$ restricts to $(\Fs,\nabla)$ on~$X$.
\end{enumerate}
Here, since~$X$ is quasi-smooth, the analytic curve~$X_{0}$ may be assumed to be quasi-smooth too.

\begin{corollary}
Theorem~\ref{thm:continuousandfinite} holds if $(\Fs,\nabla)$ is overconvergent.
\end{corollary}

\nocite{}
\bibliographystyle{alpha}
\bibliography{biblio}

\end{document}